\documentclass[reqno]{amsart}
\usepackage{booktabs}
\usepackage{pdflscape}
\usepackage{amssymb, latexsym, amsthm, enumitem, color, amsmath}

\usepackage{array}
\usepackage{blindtext}
\usepackage{longtable}
\usepackage{enumitem}
\usepackage{multirow}
\usepackage{multicol}
\usepackage{centernot}
\usepackage{mathtools}
\usepackage{bbm}
\usepackage{nicefrac}
\usepackage{upquote}

\usepackage[unicode=true,pdfusetitle,bookmarks=true,bookmarksnumbered=false,
bookmarksopen=false,breaklinks=true,pdfborder={0 0 0},
pdfborderstyle={},backref=false,colorlinks=true]{hyperref}
\definecolor{myblue}{rgb}{0.09,0.32,0.44} 
\hypersetup{pdfborder={0 0 0},pdfborderstyle={},colorlinks=true,linkcolor=myblue,citecolor=myblue,urlcolor=blue}

\newtheorem{thm}{Theorem}[section] 
\newtheorem*{thm*}{Theorem}

\newtheorem{defn}[thm]{Definition}

\newtheorem{lem}[thm]{Lemma}

\theoremstyle{remark}

\newtheorem*{rem*}{Remark}
\newtheorem*{rems*}{Remarks}

\newcommand\Cref[1]{{Corollary~\ref{#1}}}

\newcommand{\Heeo}{H(\substack{e,e_0\\w,v})}

\DeclareMathOperator{\Stab}{Stab}
\DeclareMathOperator{\Aut}{Aut}

\makeatletter
\def\moverlay{\mathpalette\mov@rlay}
\def\mov@rlay#1#2{\leavevmode\vtop{%
   \baselineskip\z@skip \lineskiplimit-\maxdimen
   \ialign{\hfil$\m@th#1##$\hfil\cr#2\crcr}}}
\newcommand{\charfusion}[3][\mathord]{
    #1{\ifx#1\mathop\vphantom{#2}\fi
        \mathpalette\mov@rlay{#2\cr#3}
      }
    \ifx#1\mathop\expandafter\displaylimits\fi}
\makeatother

\newlength{\tempindent} 
\newcommand{\lazyenum}{
\setlength{\tempindent}{\parindent} 
\begin{enumerate}[leftmargin=0cm,itemindent=0.7cm,labelwidth=\itemindent,labelsep=0cm,align=left,label=(\arabic*)]
\setlength{\parskip}{\smallskipamount}
\setlength{\parindent}{\tempindent}
}
\newif
\iffurther 
\furtherfalse

\title{Lipschitz harmonic functions on vertex-transitive graphs}
\author{Gideon Amir}
\email{gidi.amir@gmail.com}
\address{Bar-Ilan University, Ramat Gan 52900, Israel}
\author{Guy Blachar}
\email{guy.blachar@gmail.com}
\address{Bar-Ilan University, Ramat Gan 52900, Israel}
\author{Maria Gerasimova}
\email{mari9gerasimova@mail.ru}
\address{West\"{a}lische Wilhelms-Universit\"{a}t M\"{u}nster, 48149 M\"{u}nster, Germany}
\author{Gady Kozma}
\email{gady.kozma@weizmann.ac.il}
\address{The Weizmann Institute of Science, Rehovot 76100, Israel}

\begin{document}

\maketitle
\vspace{-2.5em} 
\begin{abstract}
We prove that every locally finite vertex-transitive graph $G$ admits a non-constant
Lipschitz harmonic function. 

\end{abstract}

\section{Introduction}

In this paper we are interested in the space of Lipschitz harmonic functions on a graph (precise definitions will be given below). The structure of the space of such functions played a crucial role in Kleiner's proof of Gromov's theorem \cite{Kleiner, Shalom-Tao, TaoBlog} and has thus attracted some attention \cite{Tointon, Meyerovitch-Yadin}.
In particular, Kleiner showed that any Cayley graph supports a non-trivial Lipschitz harmonic function. Our purpose in this paper is to generalize this fact to vertex-transitive graphs. Any Cayley graph is vertex-transitive, but the opposite is not true. In fact, there exist vertex-transitive graphs which are quite far from Cayley graphs in a precise sense \cite{Eskin}. Generalizing results from Cayley graphs to vertex-transitive graphs is sometimes challenging, see for example \cite{Trofimov}.

In this note we prove the following theorem, answering a problem of Georgakopoulos and Wendland \cite[Problem 1.1]{GW} (who also proved a partial result, see \cite[Proposition 3.1]{GW}). Another partial result was proved much earlier by Trofimov \cite{T98}, who showed that any infinite vertex-transitive graph supports a non-constant harmonic function, but with weaker control on the growth of the function.
\begin{thm}\label{t:main}
Every infinite, locally finite vertex-transitive graph $G$ admits a non-constant Lipschitz harmonic function.
\end{thm}

\section{Proof of Theorem \ref{t:main}}
We start with some standard definitions.
\begin{defn}For a function $f\colon V\to\mathbb{C}$ we define $\nabla f$ to be a function on the directed edges of the graph $V$ by $\nabla f(v,w)=f(v)-f(w)$. We define $\Delta f$ to be a function on $V$ by $\Delta f(v)=\sum_{w\sim v}(f(v)-f(w))$. 
A function $f\colon V \to \mathbb{C}$ on a graph $G=(V,E)$ is called harmonic if $\Delta f\equiv 0$. We will say that $f$ is Lipschitz if $\nabla f \in \ell^\infty(E).$

\end{defn}

With the theorem now completely defined we can start the proof, but,
before starting the proof, let us discuss shortly the issue of unimodularity. A locally compact group is called unimodular if its left and right Haar measures are identical, and a vertex-transitive graph is called unimodular if its automorphism group, with the topology of pointwise convergence, is unimodular. Unimodularity can play an important role in studying probability on vertex-transitive graphs, see for example, \cite[Chapter 12]{PeteBook}. Contrariwise, non-unimodular vertex-transitive graphs have the so-called modular function which can also aid in their analysis, see e.g.\ \cite{hutchcroft2020nonuniqueness}. What we will use below is the fact, first proved in \cite{Soardi-Woess}, that any amenable vertex-transitive graph is unimodular.

\begin{defn}
\label{def:amenable}
    A graph $G=(V,E)$ is called amenable if there exists a sequence of subsets $X_n\subseteq V$ such that 
    $$\lim_{n\to\infty}\frac{|\partial X_n|}{|X_n|}= 0,$$
    where $\partial X=\{v\in V\setminus X:\exists~ w\in X, w\sim v \}$ is the outer vertex boundary of $X$.
\end{defn}
Note that outer vertex boundary in this definition can be changed to any other type of boundary. The following fact is well-known, and we include its proof for completeness of the exposition. 

\begin{lem}
\label{lemma}
    Any regular amenable graph $G=(V,E)$ admits a sequence of functions $f_n$ such that 
    $$\|\nabla f_n\|_{\ell^2(E)}=1,~~\|\Delta f_n\|_{\ell^2(V)}\to 0.$$ 
\end{lem}
\begin{proof}
    We first show that the spectrum of the operator $\Delta$ contains 0. This is well-known (sometimes called Buser's inequality), but let us give the proof nonetheless. We take $h_n:=\frac{1}{\sqrt{|X_n|}}\cdot\mathbbm{1}_{X_n}$ to be the normalized characteristic functions of the sets~$X_n$ from Definition~\ref{def:amenable}. They satisfy 
    $$\|h_n\|_{\ell^2(V)}=1,\quad \|\Delta h_n\|_{\ell^2(V)}^2\le \frac{|\partial X_n|\cdot(\deg (G)+1)}{|X_n|}\to 0.$$
    Thus indeed $0$ is in the spectrum of $\Delta$ (as an operator on $\ell^2(V)$). Further, by the maximum principle there is no zero eigenfunction. 

    By the spectral theorem, the positive self-adjoint operator $\Delta$ is unitary equivalent to a multiplication operator on $L^2(Y,\nu)$ that multiplies by some non-negative function $F(y)$. The argument above shows that $\nu(F^{-1}((0,\varepsilon]))>0$ for any $\varepsilon>0$. Define $f_\varepsilon$ to be the map of $\mathbbm{1}_{F^{-1}((0,\varepsilon])}$ under the unitary equivalence. Then
    \begin{align*}\langle \Delta f_\varepsilon ,\Delta f_\varepsilon\rangle&=\int_{F^{-1}((0,\varepsilon])}F^2(y)\,d\nu(y)\le \varepsilon
    \int_{F^{-1}((0,\varepsilon])}F(y)\,d\nu(y)\\
    &=\varepsilon\langle f_\varepsilon ,\Delta f_\varepsilon\rangle
    \end{align*}
    where the equalities follow from the unitary equivalence.
    Since $\langle f, \Delta f\rangle\cdot2\deg(G)= \|\nabla f\|_{\ell^2(E)}^2$, this gives us a sequence of functions $f_n$ on $G$ such that 
    \begin{equation*}\|\nabla f_n\|_{\ell^2(E)}=1,~~\|\Delta f_n\|_{\ell^2(V)}\to 0.\qedhere\end{equation*}
\end{proof}

\begin{proof}[Proof of main theorem]
The non-amenable case follows from Piaggio and Lessa \cite{piaggio2016equivalence}. To be more precise, they prove that any stationary random graph for which random walk has positive entropy has an infinite dimensional space of bounded harmonic functions. Since vertex transitive graphs are stationary random graphs and any bounded function is Lipschitz, it remains to show that random walk on any non-amenable graph has positive entropy. To see this, we may assume that the random walk is lazy, and note that by Cheeger's inequality \cite[\S 7.2]{PeteBook} non-amenability implies that the spectral radius  of the random walk is strictly smaller then $1$. It follows that the transition probabilities decay exponentially in the number of steps, and therefore the entropy of the random walk grows linearly with the number of steps i.e.\ it has positive entropy. This finishes the non-amenable case.

We therefore assume that the graph $G=(V,E)$ is amenable. Our proof will be similar to the proof of Shalom and Tao \cite{Shalom-Tao, TaoBlog}. 

By a well-known theorem of Soardi and Woess \cite[Corollary 1]{Soardi-Woess}, a vertex-transitive graph $G$ is amenable if and only if its automorphism group $\Aut(G)$ is amenable and unimodular. We write $\mu$ for the Haar measure on $\Aut(G)$, which is bi-invariant by unimodularity, normalized so that $\mu(\Stab(v))=1$ for any $v\in V$.

Let us fix some vertex $o\in V$.
For any $a,b\in V$, denote by $H_{a,b}\subseteq \Aut(G)$ the set of all automorphisms of $G$ that map $a$ to $b$. We will use similar notations for multiple pairs of vertices / edges. Take automorphisms $f_{o,a},f_{b,o}\in\Aut(G)$ so that $f_{o,a}(o)=a$ and $f_{b,o}(b)=o$. Then $\Stab(o)=H_{o,o}=f_{b,o}H_{a,b}f_{o,a}$, so $\mu(H_{a,b})=\mu(\Stab(o))=1$ (since $\mu$ is bi-invariant).

Take $f_n$ from Lemma~\ref{lemma}, note that
$$\sum_{e_0=(o,v)\in E}\sum_{e\in E} |\nabla f_n(e)|^2\cdot\mu(H_{e,e_0})=\|\nabla f_n\|_{\ell^2(E)}^2=1,$$
so we may choose an edge $e_0$ with origin at $o$ such that 
$$\sum_{e\in E} |\nabla f_n(e)|^2\cdot\mu(H_{e,e_0})\ge \frac{1}{\deg(G)}$$
for infinitely many $n$.
After passing to a subsequence we can assume that it holds for all $n$.   

We define new functions $g_n$ by 
$$g_n(v)=\sum_{e\in E}C_{e,n} \int_{H_{e,e_0}}  f_n(T^{-1}v) ~d\mu(T), $$
where $C_{e,n}=\nabla f_n(e)$. By construction 
$$\nabla g_n(e_0)= \sum_{e\in E} |\nabla f_n(e)|^2\cdot\mu(H_{e,e_0})\ge\frac{1}{\deg(G)}.$$
We want to estimate $\|\nabla g_n\|_{\ell^2(E)}$ and $\|\Delta g_n\|_{\ell^2(V)}$. For this purpose define $\Heeo$ to be the set of automorphisms that take the edge $e$ to $e_0$ and the vertex $w$ to $v$. Note that $\mu(\Heeo)\le \mu(H_{e,e_0})\le 1$. Therefore
\begin{align*}
    |\Delta g_n(v)|&\le \sum_{e\in E}|C_{e,n}| \sum_{w\in V} \int_{\Heeo}  |\Delta f_n(T^{-1}v)| ~d\mu(T) \\
    &= \sum_{e\in E}\sum_{w\in V} |C_{e,n}| \cdot \mu\left(\Heeo)\right)\cdot |\Delta f_n(w)|\\
    &\le \|(C_{e,n})_e\|_{\ell^2(E)}\cdot \left\|\left(\mu\left(\Heeo\right)\right)_{e,w} \right\|_{\ell^2(V)\to \ell^2(E)}\cdot \|\Delta f_n\|_{\ell^2(V)}\\
    &\le \|\Delta f_n\|_{\ell^2(V)}\to 0.
\end{align*}
The second inequality follows by using Cauchy-Schwarz for the sequences $(C_{e,n})_e$ and $\sum_{w\in V}  \mu\left(\Heeo)\right)\cdot |\Delta f_n(w)|$  (note that the second norm is the  operator norm).
On the last step we used the Riesz-Thorin inequality $\|A\|^2\le \|A\|_1\cdot \|A\|_{\infty}$ and the fact that 
$$\sum_{e\in E}\mu\left(\Heeo\right)=\mu(H_{w,v})=1,\quad \sum_{w\in V}\mu\left(\Heeo\right)=\mu(H_{e,e_0})\le 1.$$
Define $\widehat{\nabla}$ by
$$\widehat{\nabla}f(v):=\sum_{w\sim v} |f(w)-f(v)|.$$
Denoting by $e^+$ the origin of the edge $e$, the same reasoning gives us
$$|\nabla g_n(e)|\le |\widehat{\nabla}g_n(e^+)|\le \|\widehat{\nabla}f_n\|_{\ell^2(V)}\le \sqrt{\deg(G)}\cdot \|\nabla f_n\|_{\ell^2(E)}=\sqrt{\deg(G)}.$$

Let us sum everything up. We have $\|\nabla g_n\|_{l^\infty(E)}\le \sqrt{\deg(G)}$, 
$\|\Delta g_n\|_{\ell^\infty(V)}\to 0$ and $\nabla g_n(e_0)\ge \nicefrac{1}{\deg(G)}$. After adding constant functions we can assume that $g_n(o)=0$. Since $\nabla g_n$ is bounded, $g_n(v)$ is bounded (with a bound that depends on $v$) and we can use compactness to pass to a subsequence where $g_n$ converges pointwise to some function $g$. It follows that $\|\nabla g\|_{\ell^\infty(E)}\le \sqrt{\deg(G)}$ and $\Delta g=0$. Finally, $|\nabla g(e_0)|>\nicefrac{1}{\deg G}$ so $g$ is not constant. This finishes the proof.
\end{proof}

\section*{Acknowledgements}
During this research G.A.\ and G.B.\ were supported by Israeli Science Foundation grant \#957/20. G.B.\ was also supported by the Bar-Ilan President's Doctoral Fellowships of Excellence. G.K.\ was supported by the Israel Science Foundation grant \#607/21 and by the Jesselson Foundation. M.G.\ was supported by the DFG -- Project-ID 427320536 -- SFB 1442, and under Germany's Excellence Strategy EXC 2044 390685587, Mathematics Münster: Dynamics--Geometry--Structure.
\bibliographystyle{plain}
\bibliography{lip_refs}

\end{document}